\documentclass{amsart}
\usepackage{amsmath,amsfonts,amssymb}
\usepackage{amsthm}
\usepackage{ifthen}
\usepackage{longtable}
\usepackage{array}
\usepackage{url}
\usepackage{hyperref}
\usepackage{enumerate}
\usepackage{enumitem}

\newcommand{\Z}{\mathbb{Z}}
\newcommand{\Q}{\mathbb{Q}}
\newcommand{\R}{\mathbb{R}}

\newcommand{\N}{\mathbb{N}}

\newcommand{\fdg}{\colon}
\newcommand{\eps}{\varepsilon}
\newcommand{\beistr}{,\quad}
\newcommand{\sqf}{\sqrt{5}}

\newtheorem*{theorem*}{Theorem}
\newtheorem{thm}{Theorem}
\newtheorem{lem}{Lemma}
\newtheorem{cor}{Corollary}
\newtheorem{pro}{Proposition}

\newtheorem{problem}{Problem}

\theoremstyle{remark}

\usepackage{tikz}
\usetikzlibrary{er,positioning, arrows, shapes.multipart}
\usepackage{calc}

\newcommand{\mybox}[5]{
\node[entity] (#1)[#2] {
 \begin{tabular}[t]{l p{2.3cm}}
   \textbf{#3}	& \\
   parameters: 	& \multicolumn{1}{p{2cm}}{#4} \\
   new bound: 	& #5
\end{tabular}
 };
}

\newcommand{\bIdOa}{293}
\newcommand{\bIdOn}{423}
\newcommand{\mIdIaII}{293}
\newcommand{\bIdIa}{302}
\newcommand{\bIdIn}{435}
\newcommand{\mIIdIm}{423}
\newcommand{\bIIdI}{300}
\newcommand{\mIdIIaII}{293}
\newcommand{\mIdIIaIII}{302}
\newcommand{\bIdIIa}{308}
\newcommand{\bIdIIn}{444}
\newcommand{\mIIdIIm}{435}
\newcommand{\mIIdIIaII}{300}
\newcommand{\bIIdII}{308}
\newcommand{\mIdIIIaII}{293}
\newcommand{\mIdIIIaIII}{302}
\newcommand{\mIdIIIaIV}{308}
\newcommand{\bIdIIIa}{317}
\newcommand{\bIdIIIn}{457}
\newcommand{\mIIdIIIm}{444}
\newcommand{\mIIdIIIaII}{300}
\newcommand{\mIIdIIIaIII}{308}
\newcommand{\bIIdIII}{315}
\newcommand{\mIdIVaII}{293}
\newcommand{\mIdIVaIII}{302}
\newcommand{\mIdIVaIV}{308}
\newcommand{\mIdIVaV}{317}
\newcommand{\bIdIVn}{470}
\newcommand{\mIIdIVm}{457}
\newcommand{\mIIdIVaII}{300}
\newcommand{\mIIdIVaIII}{308}
\newcommand{\mIIdIVaIV}{315}
\newcommand{\bIIdIV}{321}

\begin{document}
\title[On sums of Fibonacci numbers with few binary digits]{On sums of Fibonacci numbers with few binary digits}

\author[I. Vukusic]{Ingrid Vukusic}
\address{I. Vukusic,
University of Salzburg,
Hellbrunnerstrasse 34/I,
A-5020 Salzburg, Austria}
\email{ingrid.vukusic\char'100stud.sbg.ac.at} %todo

\author[V. Ziegler]{Volker Ziegler}
\address{V. Ziegler,
University of Salzburg,
Hellbrunnerstrasse 34/I,
A-5020 Salzburg, Austria}
\email{volker.ziegler\char'100sbg.ac.at}

\begin{abstract}
In this paper we completely solve the Diophantine equation $F_n+F_m=2^{a_1}+2^{a_2}+2^{a_3}+2^{a_4}+2^{a_5}$, where $F_k$ denotes the $k$-th Fibonacci number. In addition to complex linear forms in logarithms and the Baker-Davenport reduction method, we use $p$-adic versions of both tools.
\end{abstract}

\maketitle

\section{Introduction}

Sums of Fibonacci numbers with few binary digits have been of quite some interest lately. For instance, Bugeaud, Cipu and Mignotte \cite{Bugeaud_Cipu_Mignotte} determined all Fibonacci numbers with at most four binary digits. Bravo and Luca \cite{Bravo_Luca} found all sums of exactly two Fibonacci numbers which are a perfect power of two. Chim and Ziegler \cite{Chim_Ziegler} resolved the two Diophantine equations
\begin{equation}\label{eq:Chim_Ziegler}
	F_{n_1} + F_{n_2}=2^{a_1}+2^{a_2}+2^{a_3}
	\qquad \mbox{and} \qquad
	F_{n_1} + F_{n_2} + F_{n_3}=2^{a_1}+2^{a_2},
\end{equation}
where $F_k$ denotes the $k$-th Fibonacci number (defined by $F_0=0$, $F_1=1$ and $F_{k+2}=F_{k+1}+F_k$ for $k\geq 0$).
These problems are special cases of the following more general problem: Find all numbers that have few digits with respect to both the Zeckendorf and the binary representation. The second author 
proved a result about linear equations in members of two recurrence sequences \cite{Ziegler} which implies that such problems have only finitely many solutions. Moreover, he combined the above mentioned results and completely solved the problem for the bound $M=5$ of total number of digits.

The standard strategy for solving Diophantine equations like \eqref{eq:Chim_Ziegler} is the iterated application of linear forms in logarithms. It results in a huge upper bound for the largest variable. The huge bound is then reduced by the Baker-Davenport reduction method. 
The challenge is, that the reduction process requires a large number of computations and that the computation time grows exponentially with the number of variables. Using a usual PC and an efficient algorithm, one can solve such equations with up to six variables in a reasonable time. Seven variables would already take about a month's time.
In this paper we show that by using $p$-adic results one can skip one of the steps of the general procedure and thus solve equations with even more variables. 
The $p$-adic results in question are results on lower bounds for linear
forms in $p$-adic logarithms and an idea by Peth\H{o} and de Weger \cite{Pethoe_deWeger}. Pink and Ziegler \cite{Pink_Ziegler} used similar ideas to solve the Diophantine equation
\begin{equation*}
	F_n + F_m 
	=2^{z_1} 3^{z_2} 5^{z_3} \cdots 199^{z_{46}}.
\end{equation*}

We consider the following problem.

\begin{problem}\label{probl:diophantische_gleichung}
Find all $(n,m,a_1,a_2,a_3,a_4,a_5)\in \N_0^7$ with $n-1>m\geq2$ and 
$a_1> a_2 > a_3 > a_4 > a_5 \geq 0$ that solve the Diophantine equation
\begin{equation}\label{eq:main}
	F_n+F_m=2^{a_1}+2^{a_2}+2^{a_3}+2^{a_4}+2^{a_5}.
\end{equation}
\end{problem}

Thus the aim of this paper is to prove the following main theorem.

\begin{thm}\label{thm:main}
There are exactly 38 solutions $(n,m,a_1,a_2,a_3,a_4,a_5)\in \N_0^7$ to  Problem~\ref{probl:diophantische_gleichung}.
All solutions fulfill $n\leq 23$ and $a_1\leq 14$. A list of solutions is given in the Appendix.
\end{thm}

\section{Outline of the proof}\label{sec:outline}

We start by finding all solutions of Problem~\ref{probl:diophantische_gleichung} with $n\leq 1000$ by a brute force search. 
Then we show that there exist no solutions with $n>1000$: 

In Section \ref{sec:largeUpperBound} we use results on linear forms in logarithms to obtain a large upper bound for $n$. In fact, applying Matveev's Theorem \cite{Matveev} repeatedly, we obtain bounds of the form 
\[
	a_1-a_5 \ll (\log n)^5
	\qquad \mbox{and} \qquad
	n-m \ll (\log n)^5.
\]
Then we use a theorem of Bugeaud-Laurent \cite{Bugeaud_Laurent} to obtain an upper bound for $a_5$ of the form 
\[
	a_5 \ll (\log n)^2 \cdot (n-m).
\]
Combining these results we obtain an inequality of the form
\[
	n \ll a_1 =  a_5 + (a_1-a_5) \ll (\log n)^7 + (\log n)^5 	\ll
	(\log n)^7,
\]
which implies a (huge) upper bound for $n$.

In Section \ref{sec:reducing_bound} we reduce this huge bound.
Applying the Baker-Davenport reduction method we reduce the bounds for $a_1-a_5$ and $n-m$. Having a small bound for $n-m$ we can use a method due to Peth\H{o} and de Weger \cite{Pethoe_deWeger} that gives us a small bound for $a_5$.
Thus we have a small bound for $a_1=a_5 + (a_1-a_5)$ which immediately gives a small bound for $n$.

\section{Preliminaries}

First, recall that the Binet formula
\[
	F_n=\frac{\alpha^n-\beta^n}{\sqrt{5}} 
\]
holds for all $n \geq 0$, where
\[
	\alpha := \frac{1+\sqrt{5}}{2} , \quad
	\beta := \frac{1-\sqrt{5}}{2}.
\]
Next, note that $\alpha$ and $\beta$ have the following basic property
\begin{equation*}
	\alpha^{-1}=-\beta.
\end{equation*}

Moreover, the inequality
\begin{equation}\label{eq:inequ_alpha_Fib}
	\alpha^{n-2}\leq F_n \leq \alpha^{n-1}
\end{equation}
holds for all $n\geq 1$. 

Before we present some results about linear forms in logarithms, we recall a definition and some facts about logarithmic heights. Let $\gamma$ be an algebraic number of degree $d\geq 1$ with the minimal polynomial
\[
	a_d X^d + \dots +a_1 X + a_0 
	= a_d \prod _{i=1}^{d} (X- \gamma_i),
\]
where $a_0, \dots, a_d$ are relatively prime integers and $\gamma_1, \dots, \gamma_d$ are the conjugates of $\gamma$. Then the \textit{logarithmic height} of $\gamma$ is defined by
\[
	h(\gamma)
	:= \frac{1}{d} \left(
		\log |a_d|
		+ \sum_{i=1}^d \log \left( \max \{ 1,|\gamma_i|\} \right)
		\right).
\]
For any algebraic numbers $\gamma_1, \dots, \gamma_n, \gamma$ and $l\in \Z$ we have the following well-known properties.
	\begin{itemize}
	\item $h(\gamma_1 \cdots \gamma_n)\leq h(\gamma_1) + \dots + h(\gamma_n)$.
	\item $h(\gamma_1+\dots + \gamma_n)\leq  h(\gamma_1) + \dots +h(\gamma_n) + \log n$.
	\item $h(\gamma^l)=|l| h(\gamma)$.
	\end{itemize}

At the present time, the most widely used estimate for linear forms in complex logarithms is due to Matveev \cite{Matveev}. 
One of its consequences is the following \cite[Thm. 9.4]{Matveev_Folgerung}.

\begin{thm}[Matveev]\label{thm:matveev}
 Let $\gamma_1,\dots , \gamma_t$ be positive real algebraic numbers in a real number field $K$ of degree $D$, let $b_1,\dots,b_t \in \Z$, and let
\[
	\Lambda:= \gamma_1^{b_1}\cdots \gamma_t^{b_t}-1
\]
be non-zero. Then
\[
	|\Lambda|>\exp \left(
		-1.4 \cdot 30^{t+3} \cdot 4^{4.5}\cdot D^{t+2}
		(1+\log D)(1+\log B) A_1 \cdots A_t
	\right),
\]
where
\[
	B\geq \max \left\{|b_1|,\dots,|b_t|\right\}
\]
and
\[
	A_i\geq \max\left\{h(\gamma_i),\frac{|\log \gamma_i|}{D},\frac{0.16}{D}\right\}
	\quad
	\mbox{for} \quad i=1,\dots,t.
\]
\end{thm}

As for linear forms in $p$-adic logarithms, we will use a result for only two linear forms in logarithms by Bugeaud and Laurent \cite[Cor. 1]{Bugeaud_Laurent}.

\begin{thm}[Bugeaud-Laurent]\label{thm:bugeaud-laurent}
	Let $p$ be a prime number,  
$\gamma_1, \gamma_2$
mulitiplicatively independent algebraic numbers with $v_p(\gamma_1)=v_p(\gamma_2)=0$ and let $b_1,b_2 \in \N$. Then
\[
	v_p(\gamma_1^{b_1}-\gamma_2^{b_2})
	\leq
	\frac{24p(p^f-1)}{(p-1)(\log p)^4}D^4
	B^2
	A_1 A_2,
\]
where
\[
	B=
	\max \left\{
		\log b' + \log \log p + 0.4, 
		\frac{10 \log p}{D},
		10
	\right\},
\]
the number $f$ is the residue class degree of the extension $\Q_p(\gamma_1,\gamma_2)/\Q_p$, $D:=\frac{[\Q_p(\gamma_1,\gamma_2):\Q_p]}{f}$, $A_1,A_2$ are real numbers such that
\[
	 A_i \geq \max \left\{h(\gamma_i),		\frac{\log p}{D}\right\} \quad \mbox{for} \quad i=1,2
\]
and
\[
	b':=\frac{b_1}{D A_2}+\frac{b_2}{D A_1}.
\]
\end{thm}

In order to apply Theorem~\ref{thm:bugeaud-laurent} we will have to do some computations in $\Z[\frac{1+\sqrt{5}}{2}]$. We will use the following results.

\begin{lem}\label{lem:v2_q_hoch_x_plus1}
	Let $q \in \Z[\frac{1+\sqrt{5}}{2}] $. Then for any $x\in \N$ we have
	\[
	v_2(q^x+1) \in \left\{0,1,v_2(q+1),v_2(q^3+1)\right\}.
\] 
\end{lem}
\begin{proof}
Since the residue ring $\Z[\frac{1+\sqrt{5}}{2}]/(2)$ is represented by $0,1,\alpha,$ and $\beta$, we consider four cases.

\noindent \textbf{Case 1:} If $q \equiv 0 \pmod{2}$, then
obviously $v_2(q^x+1)=0$.

\noindent \textbf{Case 2:}  In the case that $q \equiv 1 \pmod{2}$, we distinguish two subcases, namely $x$ is even and $x$ is odd.

If $x$ is even, then $x=2l,\ l \in \N$.
Since $q \equiv 1 \pmod{2}$, we have $q=2k+1$ for some $k\in \Z[\frac{1+\sqrt{5}}{2}]$. By the binomial theorem we can write
\[ 
	q^l
	= 2^lk^l + \dots + {l\choose 2}\cdot 2^2k^2 + l\cdot2k + 1
	= 2\tilde{k}+1,
\]
for some $\tilde{k}\in \Z[\frac{1+\sqrt{5}}{2}]$. 
Now we have $q^x=q^{2l}=4\tilde{k}^2+4\tilde{k}+1$, so $q^x+1=4\tilde{k}^2+4 \tilde{k} + 2$ and thus $v_2(q^x+1)=1$.

If $x$ is odd, then we set
$v_2(q+1)=:l\geq 1$ (by assumption of Case 2) and write $q=2^l \cdot k -1$ for some $k \notin (2)$. By the binomial theorem we have
\[
	q^x=2^{lx}k^x-+ \dots - {x\choose 2}\cdot 2^{2l}k^2 + x \cdot 2^l\cdot k -1.
\]
Since $x$ is odd and $l \geq 1$, it follows that $v_2(q^x+1)=l=v(q+1)$.

\noindent \textbf{Case 3:} We assume that $q \equiv \alpha \pmod{2}$. Note that mod 2 we have $\alpha^1\equiv \alpha$, $\alpha^2\equiv \beta$, $\alpha^3\equiv 1$, $\alpha^4\equiv \alpha$ and so on. Therefore we consider three subcases, namely $x=3k+r$, $r = 1,2,0$.

If $x=3k+1$, then $q^x \equiv\alpha^x\equiv \alpha \pmod{2}$, so  $v_2(q^x+1)=v_2(\alpha+1)=0$.

If $x=3k+2$, then analogously $v_2(q^x+1)=v_2(\beta +1)=0$.

If $x=3k$, then we set $\tilde{q}:=q^3$. Now $\tilde{q} \equiv \alpha^3 \equiv 1 \pmod{2}$ so according to Case~2 we have either $v_2(\tilde{q}^k+1)=1$ (if $k$ is even) or $v_2(\tilde{q}^k+1)=v_2(\tilde{q}+1)$ (if $k$ is odd). Thus, 
\[
v_2(q^x+1)=
\begin{cases}
1 \beistr &$if $k$ is even,$\\
v_2(q^3+1) \beistr &$if $k$ is odd.$
\end{cases}
\]

\noindent \textbf{Case 4:} The case $q\equiv \beta \pmod{2}$ is completely analogous to Case 3.
\end{proof}

A simple consequence of Lemma~\ref{lem:gleichung_alpha_hoch_x_plus_y} is the following.

\begin{cor}\label{cor:v2_a_hoch_n_plus1}
	For any $x \in \N$ we have $v_2(\alpha^x+1)=v_2(\beta^x+1) \in \{0,1\}$.
\end{cor}

For our computations we also need estimates for $v_2(q^x-1)$:

\begin{lem}\label{lem:v2_q_hoch_x_minus1}
	Let $q \in \Z [ \frac{1+\sqrt{5}}{2} ] $ satisfy $v_2(q-1)>1$. Then for any $x\in \N$ we have 
	\[
	v_2(q^x-1)=v_2(q-1)+v_2(x).
\]
\end{lem}
\begin{proof}
The proof is analogous to the proof for integers, see e.g. \cite[Lem. 2.1.22]{Cohen}.
\end{proof}

\begin{cor}\label{cor:v2_alphahochx_minus1_xgerade}
	If $x \in \N$  is even, then
\[
	v_2(\alpha^x-1)=
	\begin{cases}
		v_2(x)+1  &$if$\ x \equiv 0 \pmod{3},\\
		0 &$otherwise$.
		
	\end{cases}
\]
\end{cor}

For the reduction of the bounds we will use Theorem \ref{thm:baker-davenport} stated below.
It is based on a lemma from the Baker-Davenport reduction method \cite{BakerDavenport}. 
Theorem~\ref{thm:baker-davenport} is a reformulation of Lemma 1 in the paper of Bravo and Luca \cite{Bravo_Luca}, which is an immediate variation of a result due to Dujella and Peth\H{o} \cite[Lemma 5]{Dujella_Pethoe}. 
In the following we denote by $\left\|x \right\|=\min\{|x-n|\fdg n \in \Z\}$ the distance to the nearest integer of $x \in \R$.

\begin{thm}\label{thm:baker-davenport}
	Let $M$ be a positive integer, let $\gamma \in \R$ be irrational and $\frac{p}{q}$ a convergent of the continued fraction of $\gamma$ such that $q>6M$. Furthermore let $A$, $B$, $\mu$ be some real numbers with $A>1$ and $B>1$. Let 
	$\eps := \left\| \mu q\right\| - M \left\|\gamma q \right\|$. If $\eps >0$ and
\[
	|u \gamma - v + \mu| < A B^{-w}
\] 
is satisfied by some $ u,v,w \in \N $ with $ u\leq M$, then
it follows that
\[
	w < \frac{\log(Aq/\eps)}{\log B}.
\]
\end{thm}

Now we consider Problem~\ref{probl:diophantische_gleichung} and establish a relation between $n$ and $a_1$.
Combining the Diophantine equation \eqref{eq:main} with the right inequality of \eqref{eq:inequ_alpha_Fib}, we get
\begin{equation*}
	2^{a_1}
	< 2^{a_1}+\dots+2^{a_5}
	= F_n + F_m
	\leq \alpha^{n-1}+\alpha^{m-1}
	< \alpha^{n-1}+\alpha^{n-2}
	=\alpha^n
	< 2^n
\end{equation*}
and thus 
\begin{equation}\label{eq:ineq_a_kleiner_n}
	a_1<n.
\end{equation}
On the other hand, the application of the left hand side of equation \eqref{eq:inequ_alpha_Fib} yields
\begin{equation*}
	\alpha^{n-2}
	\leq F_n
	< F_n+F_m
	=2^{a_1}+\dots+2^{a_5}
	< 2^{a_1}\cdot 2,
\end{equation*}
which implies
\begin{equation}\label{eq:ineq_2hocha_groesser_alphahochn}
	2^{a_1} > 1/6 \ \alpha^n.
\end{equation}

\begin{pro}\label{pro:loesungen}
There are exactly 38 solutions $(n,m,a_1,a_2,a_3,a_4,a_5)\in \N_0^7$ to  Problem~\ref{probl:diophantische_gleichung} with $n\leq 1000$.
All solutions fulfill $n\leq 23$ and $a_1\leq 14$. The list of solutions is given in the appendix.
\end{pro}
\begin{proof}
The solutions were found by a
brute force search with Sage \cite{sagemath}. It took less than a minute on a usual PC.
\end{proof}

Because of Proposition \ref{pro:loesungen}, from now on we may assume that $n>1000$.

\section{Obtaining a large upper bound for $n$}\label{sec:largeUpperBound}

In order to obtain an upper bound for $n$, we start this section by finding bounds for $n-m$ and $a_1-a_5$. As explained in Section \ref{sec:outline}, we then find a bound for $a_5$ which will finally give us a bound for $n$.

\begin{pro}\label{pro:Matveev}
	Assume that $(n,m,a_1,a_2,a_3,a_4,a_5)$ is a solution to Problem~\ref{probl:diophantische_gleichung}. Then we have 
	$n-m<3.22 \cdot 10^{65}\cdot(\log n)^5$ and $a_1-a_5<3.22 \cdot 10^{65}\cdot(\log n)^5$.
\end{pro}

\begin{proof}
We follow the strategy of Bravo and Luca \cite{Bravo_Luca} and Chim and Ziegler \cite{Chim_Ziegler} and
apply Matveev's theorem repeatedly in 5 steps.
In each step, we obtain a bound for one of the expressions $n-m, a_1-a_2, \dots , a_1-a_5$. 
The order, in which we obtain these bounds, depends on the size of $n-m$ compared to the sizes of $a_1-a_k$ for $k=2,\ldots,5$. This gives us multiple cases to consider, but we handle them simultaneously.
In each step we find ourselves in one of the following two situations.
\begin{itemize}
\item \textit{Situation 1:} 
We have bounded $a_1-a_2 , \dots , a_1-a_k$ but not yet $n-m$ ($1 \leq k \leq 5$, and $k=1$ means that we have no bounds at all).
\item \textit{Situation 2:} 
We have bounded $n-m$ and $a_1-a_2, \dots, a_1-a_k$ ($1\leq k \leq 4$).
\end{itemize}
In each Situation we proceed as follows.
\\

\noindent \textit{Situation 1:}
We start by rewriting equation \eqref{eq:main} as
\begin{equation*}
	\frac{\alpha^n-\beta^n}{\sqrt{5}} +
	\frac{\alpha^m-\beta^m}{\sqrt{5}}
	= 2^{a_1}+2^{a_2}+2^{a_3}+2^{a_4}+2^{a_5}.
\end{equation*}
We consider terms involving $n, a_1, \dots , a_k$ to be ``large''. Collecting these terms on the left hand side we obtain
\begin{equation*}
	\frac{\alpha^n}{\sqrt{5}}
	-(2^{a_1}+ \dots + 2^{a_k})
	=
	-\frac{\alpha^m}{\sqrt{5}}
	+\frac{\beta^n+\beta^m}{\sqrt{5}}
	+ 2^{a_{k+1}}+\dots+2^{a_5}.
\end{equation*}
Note that if $k=5$, the sum on the right hand side is empty. For the further computations we formally set $a_6:=-\infty$ and define $2^{-\infty}:=0$.
We take absolute values and estimate:
\begin{align*}
	\left| 
	\frac{\alpha^n}{\sqrt{5}}
	-2^{a_1}(1 + \dots + 2^{a_k-a_1})
	\right|
	&\leq
	\frac{\alpha^m}{\sqrt{5}}
	+\frac{|\beta|^n+|\beta|^m}{\sqrt{5}}
	+ 2^{a_{k+1}}(1 +\dots+2^{a_5-a_{k+1}})\\
	&<
	\frac{\alpha^m}{\sqrt{5}}
	+0.5
	+ 2 \cdot 2^{a_{k+1}}
	.
\end{align*}
Division by $2^{a_1}(1 + \dots + 2^{a_k-a_1})$ yields
\[
	\left| 
	\frac{\alpha^n}
		{\sqrt{5}\ 2^{a_1}(1+ \dots + 2^{a_k-a_1})}
	-1
	\right|
	<
	\frac{\alpha^m/\sqrt{5}+0.5 + 2 \cdot 2^{a_{k+1}}}{2^{a_1}(1+ \dots + 2^{a_k-a_1})}.
\]
We estimate the right hand side using inequality \eqref{eq:ineq_2hocha_groesser_alphahochn}
\begin{align*}
	\frac{\alpha^m/\sqrt{5}+0.5 + 2 \cdot 2^{a_{k+1}}}{2^{a_1}}
	&<
	\frac{\alpha^m (1/\sqrt{5}+0.5) + 2 \cdot 2^{a_{k+1}}}{2^{a_1}}\\
	&\leq
	\frac{\alpha^m}{2^{a_1}}
	+ \frac{2 \cdot 2^{a_{k+1}}}{2^{a_1}}\\
	&<
	\frac{\alpha^m}{\alpha^n /6}
	+ 2\cdot 2^{-(a_1-a_{k+1})}\\
	&=
	6 \cdot \alpha^{-(n-m)}
	+ 2 \cdot 2^{-(a_1-a_{k+1})} \\
	&\leq
	8 \cdot \max \left\{ \alpha^{-(n-m)},\ 2^{-(a_1-a_{k+1})} \right\}
\end{align*}
and thus obtain
\begin{equation}\label{eq:ungl:matveev_sit1_zwischenergebnis}
	\left| 
	\frac{\alpha^n}
		{\sqrt{5}\ 2^{a_1}(1+ \dots + 2^{a_k-a_1})}
	-1
	\right|
	<
	8 \cdot \max \left\{ \alpha^{-(n-m)},\ 2^{-(a_1-a_{k+1})} \right\}.
\end{equation}
The left hand side of \eqref{eq:ungl:matveev_sit1_zwischenergebnis} 
is not zero, since if it were zero we would get that $\alpha^n=\sqrt{5}\ 2^{a_1}(1+ \dots + 2^{a_k-a_1})$ and thus $\alpha^{2n}\in \Z$, which is impossible for all $n\geq 1$.
Therefore we can apply Matveev's theorem by taking parameters $t:=3$ and 
\begin{align*}
	\gamma_1 &:=\alpha \beistr 
	&& b_1:=n \beistr \\
	\gamma_2 &:=2 \beistr 
	&& b_2:=-a_1 \beistr \\
	\gamma_3 &:= \sqrt{5}(1+\dots +2^{a_k-a_1}) \beistr 
	&& b_3:=-1.
\end{align*}
The three numbers $\gamma_1, \gamma_2, \gamma_3$ are real, positive and belong to $K := \Q(\sqrt{5})$, so we can take $D:=2$.
By inequality \eqref{eq:ineq_a_kleiner_n} we have $\max \left\{|b_1|,|b_2|,|b_3|\right\}=n$ and we can take $B:=n$.
Since $h(\gamma_1)=(\log \alpha)/2=0.2406\ldots$ and $h(\gamma_2)=\log 2=0.6931\ldots$, we can choose $A_1:=0.25$ and $A_2:=0.7$. In order to choose $A_3$, we estimate 
the logarithmic height of $\gamma_3$. Using the properties of logarithmic heights and keeping in mind that $a_1-a_2 < \dots < a_1-a_k$ and
$k\leq5$ we get 
\begin{align*}
	h(\gamma_3)
	=h\left(\sqrt{5}\left(1+\dots +2^{a_k-a_1}\right)\right)
	&\leq
	h\left(\sqrt{5}\right) + h\left(1+\dots +2^{a_k-a_1}\right)\\
	&\leq
	\log \sqrt{5} + h(1)+ \dots + h\left(2^{a_k-a_1}\right) + \log k \\
	& \leq
	\log \sqrt{5} + 4 (a_1-a_k) \cdot h(2) + \log 5\\
	&=
	\log \sqrt{5}+ 4 \cdot (a_1-a_k) \cdot \log 2+ \log 5 \\
	&\leq
	(a_1-a_k)\cdot\left(4 \log 2 + \log \sqrt{5}+ \log 5\right)\\ 
	&< 5.19 \cdot (a_1-a_k)
	.
\end{align*}
and we choose $A_3:= 5.19 \cdot \max\{a_1-a_k,1\}$. 
Application of Matveev's theorem to the left hand side of 
 \eqref{eq:ungl:matveev_sit1_zwischenergebnis} yields
\begin{align*}
	\exp\left(-1.4 \cdot 30^6 \cdot 3^{4.5} \cdot 2^5 (1+\log 2)(1+\log n)
	 \cdot 0.25 \cdot 0.7 \cdot 5.19 \cdot (a_1-a_k)
	 \right)\\
	 < 
	 8 \cdot \max \left\{ \alpha^{-(n-m)},\ 2^{-(a_1-a_{k+1})} \right\}
	.
\end{align*}
Since $n>1000$, we can
estimate $1+\log n < 1.15 \log n$ obtaining
\begin{equation*}
	\exp \left(- C_1 \log n \ (a_1-a_k)\right)
	< 
	8 \cdot \max \left\{ \alpha^{-(n-m)},\ 2^{-(a_1-a_{k+1})} \right\},
\end{equation*}
where $C_1:=8.11\cdot 10^{12}>1.4 \cdot 30^6 \cdot 3^{4.5} \cdot 2^5 \cdot (1+\log 2) \cdot 1.15 \cdot 0.25 \cdot 0.7 \cdot 5.19 
$.
Taking logarithms we get
\begin{equation*}
	\min\{(n-m)\log \alpha, (a_1-a_{k+1})\log 2\}
	<
	C_1 \log n \cdot (a_1-a_k)+\log 8.
\end{equation*}
Since we are assuming that $n>1000$ and $a_1-a_k\geq1$ and we estimated $C_1$ quite roughly, we may now drop the small constant on the right hand side and obtain
\begin{equation*}
	\min\{(n-m)\log \alpha, (a_1-a_{k+1})\log 2\}
	 <  
	 C_1 \log n \cdot (a_1-a_k).
\end{equation*}

Thus, in Situation 1 we get a bound of the form $\frac{C_1}{c} \log n \cdot (a_1-a_k)$, where $C_1=8.11 \cdot 10^{12}$ and the value of $c$ depends on the case. If 
$\alpha^{-(n-m)} \leq 2^{-(a_1-a_{k+1})}$, then $c=\log 2$ and the given bound is a bound for $a_1-a_{k+1}$. If $\alpha^{-(n-m)} > 2^{-(a_1-a_{k+1})}$, then $c=\log \alpha$ and the given bound is a bound for $n-m$.
\\

\noindent \textit{Situation 2:}
As in Situation 1, we start by rewriting equation \eqref{eq:main} as
\begin{equation*}
	\frac{\alpha^n-\beta^n}{\sqrt{5}} +
	\frac{\alpha^m-\beta^m}{\sqrt{5}}
	= 2^{a_1}+2^{a_2}+2^{a_3}+2^{a_4}+2^{a_5}.
\end{equation*}
Now we consider $n$, $m$ and $a_1, \dots , a_k$ to be ``large'' and collecting the corresponding terms
on the left hand side we obtain
\begin{equation*}
	\frac{\alpha^n+\alpha^m}{\sqrt{5}}
	-(2^{a_1}+ \dots + 2^{a_k})
	=
	\frac{\beta^n+\beta^m}{\sqrt{5}}
	+ 2^{a_{k+1}}+\dots+2^{a_5}.
\end{equation*}
We take absolute values and estimate:
\begin{align*}
	\left|
	\frac{\alpha^n+\alpha^m}{\sqrt{5}}
	-2^{a_1}(1 + \dots + 2^{a_k-a_1})
	\right|
	&\leq
	\frac{|\beta|^n+|\beta|^m}{\sqrt{5}}
	+ 2^{a_{k+1}}(1+\dots+2^{a_5-a_{k+1}})\\
	&<
	0.5 + 2\cdot 2^{a_{k+1}}.
\end{align*}
Division by $2^{a_1}(1 + \dots + 2^{a_k-a_1})$ yields
\[
	\left| 
	\frac{\alpha^n+\alpha^m}
		{\sqrt{5}\ 2^{a_1}(1+ \dots + 2^{a_k-a_1})}
	-1
	\right|
	<
	\frac{0.5 + 2\cdot 2^{a_{k+1}}}{2^{a_1}(1+ \dots + 2^{a_k-a_1})},
\]
which we estimate further and write
\begin{equation}\label{eq:ungl_matveev_sit2}
	\left|
	\frac{\alpha^n(1 +\alpha^{m-n})}
		{\sqrt{5}\ 2^{a_1}(1+ \dots + 2^{a_k-a_1})}
	-1
	\right|
	<
	\frac{2.5 \cdot 2^{a_{k+1}}}{2^{a_1}}
	=
	2.5 \cdot 2^{-(a_1-a_{k+1})}
	.
\end{equation}
As in Situation 1, one can check that the left hand side is not equal to zero, so we can apply Matveev's theorem.
We take the parameters $t:=3$ and
\begin{align*}
	\gamma_1 &:=\alpha \beistr 
	&& b_1:=n \beistr \\
	\gamma_2 &:=2 \beistr 
	&& b_2:=-a_1 \beistr \\
	\gamma_3 &:= \frac{1+\alpha^{m-n}}{\sqrt{5}(1+\dots +2^{a_k-a_1})} \beistr 
	&& b_3:=1.
\end{align*}
As before, we take $D:=2$ and $B:=n$. Also, we keep $A_1:=0.25$ and $A_2:=0.7$. For $A_3$ we estimate the logarithmic height of $\gamma_3$. Using the properties of logarithmic heights and keeping in mind that $k\leq4$ we get
\begin{align*}
	h(\gamma_3)
	& =h\left(\frac{1+\alpha^{m-n}}{\sqrt{5}(1+\dots + 2^{a_k-a_1})}\right)\\
	&\leq 
	h\left(1+\alpha^{m-n}\right)+h\left(\sqrt5\right)+h\left(1+\dots +2^{a_k-a_1}\right)\\
	&\leq
		h(1)+h\left(\alpha^{m-n}\right)+\log 2 
		+ \log{\sqrt{5}}
		+ h(1)+ \dots + h\left(2^{a_k-a_1}\right) + \log k\\
	&\leq
		(n-m)h(\alpha)
		+ \log 2 + \log\sqrt{5}
		+ 3 \cdot (a_1-a_k) h(2)
		+ \log 4 \\
	&=
		(n-m) \frac{\log \alpha}{2}
		+ \log 2 + \log \sqrt{5}
		+ 3 \cdot (a_1-a_k) \log 2  
		 + \log 4 \\
	&=
		(n-m) \frac{\log \alpha}{2}
		+ 3 \cdot (a_1-a_k) \log 2 
		+ \log \left(8 \sqrt{5}\right) \\
	&\leq
		\max \left\{ n-m,\ a_1-a_k \right\} 
		\cdot \left(\frac{\log \alpha}{2} + 3 \log 2 + \log \left(8 \sqrt{5}\right) \right) \\
	&< 5.21 \cdot \max \left\{ n-m,\ a_1-a_k \right\} 
	.
\end{align*}
Thus we set $A_3:= 5.21 \cdot \max \left\{ n-m,\ a_1-a_k \right\}$ and Matveev's theorem 
combined with inequality \eqref{eq:ungl_matveev_sit2} yields similarly as in Situation~1
\begin{align*}
	\exp \left(-C_2' \cdot
	\log n \cdot
	\max \left\{ n-m,\ a_1-a_k \right\}
	\right)
	<
	2.5 \cdot 2^{-(a_1-a_{k+1})}
	,
\end{align*}
where $C_2':=8.134 \cdot 10^{12} > 1.4 \cdot 30^6 \cdot 3^{4.5} \cdot 2^5 (1+\log 2) \cdot 1.15 \cdot 0.25 \cdot 0.7 \cdot 5.21$. 
Taking logarithms we obtain
\[
	a_1-a_{k+1}
	< 
	\frac{C_2'}{\log 2} \cdot \log n \cdot \max\left\{ n-m,a_1-a_k \right\}
	+ \frac{\log 2.5}{\log 2}.
\]
Since we are assuming that $n>1000$ and $\max\left\{ n-m,a_1-a_k \right\} \geq 1$, we obtain the bound
\begin{equation*}
	a_1-a_{k+1}
	< 
	\frac{C_2}{\log 2} \cdot \log n \cdot \max\left\{ n-m,a_1-a_k \right\},
\end{equation*}
with $C_2=8.14\cdot 10^{12}$.
\\

\noindent \textit{The five Steps:}
At the beginning, we have no bounds yet, so we start in Situation 1 and after Step 1 we have obtained either a bound for $n-m$ or $a_1-a_2$. In each further step we are either in Situation 1 or Situation 2. In either Situation, we obtain a result of the form
\[
	T < \frac{C}{c} \cdot \log n \cdot T',
\]
where $T\in \left\{ n-m, a_1-a_2, \dots , a_1-a_5\right\}$ is an expression that we have not bounded yet and $T'$ is either equal to 1 (in Step 1) or $T' \in \left\{ n-m, a_1-a_2, \dots , a_1-a_5\right\}$ is an
expression, that we have bounded in one of the previous steps. 
Furthermore, $c=\log \alpha$ if $T=n-m$, and $c=\log 2$ otherwise.
For the constant $C$ we can take $C:=8.14=\max\{C_1,C_2\}$.

Depending on how large $2^{-(n-m)}$ is compared to $\alpha^{-(a_1-a_k)}$ for $2 \leq k \leq 5$, we get the bounds for specific expressions in specific steps. In any case, after 5 steps we have obtained bounds for all epxressions $n-m, a_1-a_2, \dots , a_1-a_5$ and the largest appearing bound is
\[
	C^5	\frac{1}{\log \alpha} \cdot \frac{1}{(\log 2)^4} \cdot (\log n)^5 
	<  3.22 \cdot 10^{65} \cdot (\log n)^5.
\]
In particular, this bound is an upper bound for every expression $n-m, a_1-a_2, \ldots , a_1-a_5$, which yields the content of Proposition \ref{pro:Matveev}.
\end{proof}

\begin{pro}\label{pro:Bugeaud}
	Assume that $(n,m,a_1,a_2,a_3,a_4,a_5)$ is a solution to  Problem~\ref{probl:diophantische_gleichung}. Then we have 
\[
	a_5<2927(\log n)^2(n-m).
\]
\end{pro}

In the proof of Proposition~\ref{pro:Bugeaud} we will apply the theorem of Bugeaud-Laurent. This will require the multiplicative independence of expressions of the form $\frac{\alpha}{\beta}$ and $\frac{\beta^x+1}{\alpha^x+1}$. Therefore, we prove two lemmas first.

\begin{lem}\label{lem:gleichung_alpha_hoch_x_plus_y}
	The only solutions $(x,y)\in \N^2$ to 
\[
	\alpha^{x+y}=\alpha^x+\alpha^y+1
\] 
are $(1,3)$ and $(3,1)$.
\end{lem}
\begin{proof}
	Suppose, without loss of generality, that $x\geq y$. Then we have
\[
	2 \alpha^{x+y-2}
	<
	\alpha^{x+y}
	=
	\alpha^x+\alpha^y+1
	\leq
	2\alpha^x +1.
\]
This implies $x+y-2\leq x$ and thus $y\leq 2$. 

If $y=1$, then the equation becomes $\alpha^{x+1}=\alpha^x+\alpha +1$ which is equivalent to $\alpha^x=\frac{\alpha+1}{\alpha-1}$. Since $\frac{\alpha+1}{\alpha-1}=\alpha^3$, we get $x=3$.

If $y=2$, then the equation becomes $\alpha^{x+2}=\alpha^x+\alpha^2 +1$, which is equivalent to $\alpha^x=\frac{\alpha^2+1}{\alpha^2-1}$. But one can check that $\alpha<\frac{\alpha^2+1}{\alpha^2-1}<\alpha^2$, so there is no solution.

Therefore, if $x\geq y$, the only solution is $(3,1)$ and for $x<y$ we get the solution $(1,3)$.
\end{proof}

\begin{lem}\label{lem:mult_unabh}
	For $x\in \N$ the expressions $\frac{\alpha}{\beta}$ and $\frac{\beta^x+1}{\alpha^x+1}$ are multiplicatively dependent if and only if $x=1$, $x=3$ or $x$ is even.
\end{lem}
\begin{proof}
	Since $\frac{\alpha}{\beta}=-\alpha^2$, we have to check for which $x$ there exist $k,l\in\Z$ not both zero such that $(-\alpha^2)^k=(\frac{\alpha^x+1}{\beta^x+1})^l$. Note that $\alpha$ is a fundamental unit in $\Z[\frac{1+\sqrt{5}}{2}]$ and, in particular, $-\alpha^2$ is a unit. 
%	So, if we have multiplicative dependency, then $\frac{\alpha^x+1}{\beta^x+1}$ must be a unit. On the other hand, every unit can be written as $\pm \alpha^t$, $t\in \Z$, and if $\frac{\alpha^x+1}{\beta^x+1}=\pm \alpha^t$ is a unit, then $(-\alpha^2)^k=(\pm \alpha^t)^l$ holds for $k=2t$ and $l=4$.
Thus, we have multiplicative dependency if and only if $\frac{\alpha^x+1}{\beta^x+1}$ is a unit.
	
	If $x$ is even, then since $\alpha^{-1}=-\beta$ we have $\beta^x=\alpha^{-x}$ and so
$
	\frac{\alpha^x+1}{\beta^x+1}
	=\alpha^x
$
is a unit.

Now let $x$ be odd. Then
$
	\frac{\alpha^x+1}{\beta^x+1}
	=\alpha^x \cdot \frac{\alpha^x+1}{\alpha^x-1}.
$
Since $\alpha^x$ is a unit, we need to check for which $x$ the expression $\frac{\alpha^x+1}{\alpha^x-1}$ is a unit, i.e. for which $x \in \N$ there exists a $y \in \Z$ such that $\frac{\alpha^x+1}{\alpha^x-1}=\pm \alpha^y$. Since $x\geq 1$, we have $\frac{\alpha^x+1}{\alpha^x-1}> 1$ and therefore we need to find all $x,y\geq1$ such that $\frac{\alpha^x+1}{\alpha^x-1}=\alpha^y$. This is equivalent to $\alpha^{x+y}=\alpha^x+\alpha^y+1$ and by Lemma~\ref{lem:gleichung_alpha_hoch_x_plus_y} the equation can be satisfied if and only if $x=1$ or $x=3$. 
\end{proof}

\begin{proof}[Proof of Proposition~\ref{pro:Bugeaud}]
We consider equation \eqref{eq:main} and rewrite it as
\begin{equation}\label{eq:dioph_rewritten}
	\frac{\alpha^m (\alpha^{n-m}+1)}{\sqrt{5}} -
	\frac{\beta^m (\beta^{n-m}+1)}{\sqrt{5}}
	= 2^{a_1}+2^{a_2}+2^{a_3}+2^{a_4}+2^{a_5}.
\end{equation}
For the right hand side we have
$v_2(2^{a_1}+\dots+2^{a_5})= a_5$.
In order to obtain a bound for $a_5$, we want to apply the theorem of Bugeaud-Laurent to the left hand side of \eqref{eq:dioph_rewritten}.
Note that $v_2(\sqrt{5})=v_2(\beta^m)=0$ and recall from Corollary~\ref{cor:v2_a_hoch_n_plus1} that $v_2(\alpha^{n-m}+1) \in \{0,1\}$. Therefore, multiplication by $\sqrt{5}$ and division by $\beta^m$ and $(\alpha^{n-m}+1)$ yield
\begin{equation}\label{eq:Bugeaud_abschaetzung}
	a_5 
	=
	v_2\left(
		\frac{\alpha^m (\alpha^{n-m}+1)}{\sqrt{5}} -
		\frac{\beta^m (\beta^{n-m}+1)}{\sqrt{5}}
	\right)	
	\leq v_2\left(
		\left(\frac{\alpha}{\beta}\right)^m-
		\frac{\beta^{n-m}+1}{\alpha^{n-m}+1}
	\right)+1.
\end{equation}
From Lemma~\ref{lem:mult_unabh} we know that $\gamma_1:=\frac{\alpha}{\beta}$ and $\gamma_2:=\frac{\beta^{n-m}+1}{\alpha^{n-m}+1}$ are multiplicatively independent if $n-m$ is odd and $n-m\neq 3$ (note that $n-m>1$). 
Also, $v_2(\gamma_1)=v_2(\gamma_2)=0$, so in these cases we can apply the theorem of Bugeaud-Laurent for $b_1=m$, $b_2=1$ and $p=2$. We have $\Q_2(\gamma_1,\gamma_2)=\Q_2(\sqrt{5})$, $f=2$ and $D=\frac{[\Q_2(\sqrt{5}):\Q_2]}{f}=1$.

In order to choose $A_1$ and $A_2$, we estimate the logarithmic heights of $\gamma_1$ and $\gamma_2$.
\begin{align*}
	h(\gamma_1)&=h\left(\frac{\alpha}{\beta}\right)
	\leq 2\, h(\alpha)= \log \alpha,\\
	h(\gamma_2)
	&=h\left(\frac{\beta^{n-m}+1}{\alpha^{n-m}+1}\right)
	\leq 2\, h ( \alpha^{n-m}+1 )\\
	&\leq 2\,((n-m)h(\alpha) + \log 2)
	=(n-m)\log \alpha + 2 \log 2.
\end{align*}
Thus we can set $A_1:=\log 2 \geq \max \left\{ h(\gamma_1), \frac{\log 2}{D} \right\}$ and $ A_2:=(n-m)\log \alpha + 2 \log 2\geq \max \left\{ h(\gamma_2), \frac{\log 2}{D} \right\}$.
Finally, we compute
\begin{align*}
	b'
	= \frac{b_1}{D A_2}+\frac{b_2}{D A_1}
	=\frac{m}{(n-m)\log \alpha+ 2 \log 2}
		+ \frac{1}{\log 2}
	< m + 1,
\end{align*}
where we took into account that $n-m\geq 2$ and $m\geq2$.
Now the theorem of Bugeaud-Laurent tells us that
\begin{align*}
	v_2\left(
		\left(\frac{\alpha}{\beta}\right)^m-
		\frac{\beta^{n-m}+1}{\alpha^{n-m}+1}
	\right)
	\leq
	&\frac{24\cdot 2 \cdot (2^2-1)}{(2-1)(\log 2)^4}\\
		&\cdot \left(\max\left\{
			\log( m + 1)+ \log \log 2 + 0.4,10\log 2,10\right\}
		\right)^2 \\
	&\cdot \log 2 \cdot ((n-m)\log \alpha + 2 \log 2).
\end{align*}
We compute, estimate and, noting that $m\geq 2$ and $n-m\geq 2$, obtain
\begin{align}\label{eq:Bugeaud_mult_unabh}
	v_2\left(
		\left(\frac{\alpha}{\beta}\right)^m-
		\frac{\beta^{n-m}+1}{\alpha^{n-m}+1}
	\right)
	&< 508 (\log(m+1)+9)^2(n-m).
\end{align}

Now we consider the cases where $\alpha_1$ and $\alpha_2$ are multiplicatively dependent, that is where $n-m=3$ or $n-m$ is even.

If $n-m=3$, then we have 
\[
	\frac{\beta^{n-m}+1}{\alpha^{n-m}+1}
	= \frac{\beta^3+1}{\alpha^3+1}
	=\alpha^{-4}
\]
and thus
\begin{align*}
	\left(\frac{\alpha}{\beta}\right)^m-
		\frac{\beta^{n-m}+1}{\alpha^{n-m}+1}
	= \left(-\alpha^2\right)^m -\alpha^{-4}
	=\alpha^{-4}\left(\pm\alpha^{2m+4}-1\right) .
\end{align*}
We note that $v_2(\alpha^{-4})=0$ and by Corollaries \ref{cor:v2_a_hoch_n_plus1} and 
 \ref{cor:v2_alphahochx_minus1_xgerade}
\begin{align*}
	v_2\left(\pm\alpha^{2m+4}-1\right)
	&\leq 
		\max \left\{
	v_2(\alpha^{2m+4}-1),v_2(-\alpha^{2m+4}-1)
		\right\}\\
	&= 
		\max \left\{
	v_2(\alpha^{2m+4}-1),v_2(\alpha^{2m+4}+1)
		\right\}\\
	&\leq
		\max \left\{
	v_2(2m+4)+1,1
		\right\}\\
	&=v_2(m+2)+2,
\end{align*}
so we obtain
\begin{equation}\label{eq:Bugeaud_gleich3}
	v_2\left(
		\left(\frac{\alpha}{\beta}\right)^m-
		\frac{\beta^{n-m}+1}{\alpha^{n-m}+1}
	\right)
	\leq 
	v_2(m+2)+2
	\leq
	\frac{\log(m+2)}{\log 2}+2.
\end{equation}

If $n-m$ is even, then
\begin{align*}
 	\left(\frac{\alpha}{\beta}\right)^m-
 		\frac{\beta^{n-m}+1}{\alpha^{n-m}+1}
	&=(-\alpha^2)^m-
		\frac{(\alpha^{-1})^{n-m}+1}{\alpha^{n-m}+1}
	=\pm \alpha^{2m} - \alpha^{-(n-m)}\\
	&= \alpha^{-(n-m)}(\pm \alpha^{n+m}-1).
\end{align*}
Analogously to the case $n-m=3$ we obtain
\begin{equation}\label{eq:Bugeaud_gerade}
	v_2\left(
		\left(\frac{\alpha}{\beta}\right)^m-
		\frac{\beta^{n-m}+1}{\alpha^{n-m}+1}
	\right)
	\leq 
	v_2(n+m)+1
	\leq
	\frac{\log (n+m)}{\log 2}+1.
\end{equation}

Finally, we compare the upper bounds from \eqref{eq:Bugeaud_mult_unabh}, \eqref{eq:Bugeaud_gleich3} and \eqref{eq:Bugeaud_gerade}. 
If we estimate  $\log(m+1)<\log n$ in \eqref{eq:Bugeaud_mult_unabh}, we obtain a bound that is grater than those in 
\eqref{eq:Bugeaud_gleich3} and \eqref{eq:Bugeaud_gerade} in any case. 
Combining this bound with \eqref{eq:Bugeaud_abschaetzung}, we obtain
\[
	a_5 < 508(\log n + 9)^2(n-m)+1.
\]

Since we are assuming that $n>1000$ and $\log 1000 \approx 6.9$, we can estimate $\log n +9 < 2.4 \log n$ and obtain
\[
	a_5 < 508(2.4 \log n)^2(n-m)+1
	< 2927 (\log n)^2(n-m).
\]
Thus Proposition~\ref{pro:Bugeaud} is proved.
\end{proof}

Now we are in the position to prove the following proposition.

\begin{pro}\label{pro:grosseSchranke}
		Assume that $(n,m,a_1,a_2,a_3,a_4,a_5)$ is a solution to Problem ~\ref{probl:diophantische_gleichung}. Then we have that $n<1.54 \cdot 10^{85}$.
\end{pro}
\begin{proof}
From inequality \eqref{eq:ineq_2hocha_groesser_alphahochn} we obtain $a_1 \log 2 > -\log 6 + n \log \alpha $. Rewriting this inequality and estimating we get $n-4<1.45 a_1$. Now we estimate further applying Propositions \ref{pro:Matveev} and \ref{pro:Bugeaud}:
\begin{align*}
	n-4 
	&< 
	1.45 \cdot (a_5 + (a_1-a_5))\\
	&<
	1.45\cdot (2927\cdot(\log n)^2(n-m+1) + 3.22 \cdot 10^{65} \cdot (\log n)^5)\\
	&<
	1.45 \cdot (2927\cdot (\log n)^2\cdot 3.22 \cdot 10^{65}\cdot (\log n)^5 + 3.22 \cdot 10^{65} \cdot (\log n)^5)\\
	&<
	1.45 \cdot 9.43 \cdot 10^{68} (\log n)^7,
\end{align*}
which gives us the inequality $n<1.37 \cdot 10^{69} (\log n)^7$. Solving this inequality we obtain $n<1.54 \cdot 10^{85}$.
\end{proof}

\section{Reducing the bound for $n$}\label{sec:reducing_bound}

In this section we first reduce the bounds for $n-m$ and $a_1-a_2, \ldots , a_1-a_5$ using the Baker-Davenport reduction method. After that, we reduce the bound for $a_5$ with a $p$-adic reduction method. This leads to a small bound for $n$. We start with the following proposition.

\begin{pro}\label{pro:Baker-Davenport}
	Assume that $(n,m,a_1,a_2,a_3,a_4,a_5)$ is a solution to Problem~\ref{probl:diophantische_gleichung}. Then we have that $n-m\leq 470$, $a_1-a_2\leq 300$, $a_1-a_3\leq 308$, $a_1-a_4\leq 315$ and $a_1-a_5\leq 321$. 
\end{pro}

\begin{proof}
The proof of Proposition~\ref{pro:Baker-Davenport} uses the Baker-Davenport reduction method in the same way as Bravo and Luca \cite{Bravo_Luca} and Chim and Ziegler \cite{Chim_Ziegler} did.
The idea is to repeat the steps from the proof of Proposition~\ref{pro:Matveev}, but  instead of Matveev's Theorem we apply Theorem~\ref{thm:baker-davenport} obtaining a new smaller bound for one of the expressions $n-m, a_1-a_2, \dots, a_1-a_5$ in each step. As before, in each step we find ourselves in one of the following two situations.
\begin{itemize}
\item \textit{Situation 1:} We have found small bounds $a_1-a_2 , \dots , a_1-a_k$ ($1 \leq k \leq 5$) but not yet for $n-m$.
\item \textit{Situation 2:} We have found small bounds for $n-m$ and $a_1-a_2, \dots, a_1-a_k$ ($1\leq k \leq 4$).
\end{itemize}
We consider each situation separately.
\\

\noindent \textit{Situation 1:}
We recall inequality \eqref{eq:ungl:matveev_sit1_zwischenergebnis} from Situation 1 in the proof of Proposition~\ref{pro:Matveev}:
\[
	\left|
		\frac{\alpha^n}
		{\sqrt{5}\ 2^{a_1}(1+ \dots + 2^{a_k-a_1})}
	-1
	\right|
	<
	8 \cdot \max \left\{ \alpha^{-(n-m)},2^{-(a_1-a_{k+1})} \right\}.
\]
We set $\Lambda:=\log \left( 
		\frac{\alpha^n}
		{\sqrt{5}\ 2^{a_1}(1+ \dots + 2^{a_k-a_1})}
		\right)$,
so we have 	
\[
	\left| e^\Lambda -1 \right| \leq
	8 \cdot \max \left\{ \alpha^{-(n-m)},2^{-(a_1-a_{k+1})} \right\}.
\]
Note that the inequality $|x|\leq 2 |e^x-1|$ holds for all $x\geq -1$.
If $\Lambda<-1$, then 
$0.63 < \left| e^{\Lambda}-1 \right| \leq
	8 \cdot \max \left\{ \alpha^{-(n-m)},2^{-(a_1-a_{k+1})} \right\}$, which implies 
$ \min\{(n-m)\log \alpha, (a_1-a_{k+1})\log 2\} <\log 8 - \log 0.63 <2.55$ and we are done. 
If $\Lambda\geq -1$, then we have
$\frac{1}{2}|\Lambda| 
\leq 
\left| e^{\Lambda}-1 \right| 
\leq 8 \cdot \max \left\{ \alpha^{-(n-m)},2^{-(a_1-a_{k+1})} \right\}
$,
which implies
\begin{multline*}
	\left|
	n \log \alpha -a_1 \log 2 - 
	\log\left(\sqrt{5}(1+\dots + 2^{a_k-a_1})\right)
	\right| \\
	<
	16 \cdot \max \left\{ \alpha^{-(n-m)},2^{-(a_1-a_{k+1})} \right\}.
\end{multline*}
Division by $\log 2$ yields
\begin{multline*}
	\left|
	n \cdot \frac{ \log \alpha}{\log 2} -a_1 +
	\frac{-\log\left(\sqrt{5}(1+\dots + 2^{-(a_1-a_k)})\right)}{\log 2}
	\right| \\
	<
	23.09 \cdot \max \left\{ \alpha^{-(n-m)},2^{-(a_1-a_{k+1})} \right\}.
\end{multline*}

Distinguishing between different cases, we then apply Theorem~\ref{thm:baker-davenport} by setting 
\begin{align*}
	M &:=1.54 \cdot 10^{85}>n, \\
	\gamma &:=\frac{\log \alpha}{\log 2}, \\
	\mu&=\mu_{a_1-a_2,\dots,a_1-a_k} :=\frac{-\log\left(\sqrt{5}(1+\dots + 2^{a_k-a_1})\right)}{\log 2}, \\
	A &:= 23.09 \quad \mbox{and} \\
	B &:=2 \quad \mbox{or} \quad B:= \alpha,
\end{align*}
where the value of $B$ depends on whether $2^{-(a_1-a_{k+1})}\geq \alpha^{-(n-m)}$ or $2^{-(a_1-a_{k+1})}< \alpha^{-(n-m)}$.
Note that $\gamma$ is irrational because $\alpha$ and $2$ are multiplicatively independent.
Depending on $k$ and the bounds for $a_1-a_2, \dots , a_1-a_k$ we have to check a certain number of $\mu$'s.
For each $\mu$ we search for a convergent $p/q$ to $\gamma$ with $q>6M$  such that $\eps=\left\|\mu q \right\| - M \left\|\gamma q \right\| >0$. Then we obtain $a_1-a_{k+1} \leq \frac{\log (Aq/\eps)}{\log B}$ or $n-m \leq \frac{\log (Aq/\eps)}{\log B}$. The maximum of these bounds for all $\mu$'s is a bound for $a_1-a_{k+1}$ or $n-m$ respectively.
\\

\noindent \textit{Situation 2:}
Suppose we have found small bounds for $n-m$ and $a_1-a_2, \dots , a_1-a_k$ ($1\leq k \leq 4$).
We recall inequality \eqref{eq:ungl_matveev_sit2} from Situation 2 in the proof of Proposition~\ref{pro:Matveev}:
\[
	\left|
		\frac{\alpha^n(1 +\alpha^{m-n})}
		{\sqrt{5}\ 2^{a_1}(1+ \dots + 2^{a_k-a_1})}
	-1
	\right|
	<
	2.5 \cdot 2^{-(a_1-a_{k+1})}.
\]
Similarly as in Situation 1, we set $\Lambda:=\log \left( 
		\frac{\alpha^n(1 +\alpha^{m-n})}
		{\sqrt{5}\ 2^{a_1}(1+ \dots + 2^{a_k-a_1})}
		\right)$ and obtain
\begin{equation}\label{eq:BD_Ungl_Sit2}
	\left|
	n \cdot \frac{ \log \alpha}{\log 2} -a_1 + 
		\frac{		
		\log\left(
			\frac{1+\alpha^{-(n-m)}}{\sqrt{5}\left(1+\dots + 2^{-(a_k-a_1)}\right)}
		\right) }{\log 2}
	\right|
	<
	7.22 \cdot 2^{-(a_1-a_{k+1})}.
\end{equation}
As in Situation 1, we then try to apply Theorem~\ref{thm:baker-davenport} by setting 
\begin{align*}
	M &:=1.54\cdot 10^{85}>n, \\
	\gamma &:=\frac{\log \alpha}{\log 2}, \\
	\mu&=\mu_{m-n,a_1-a_2,\dots,a_1-a_k}:=
			\frac{\log\left(
			\frac{1+\alpha^{-(n-m)}}{\sqrt{5}(1+\dots + 2^{-(a_1-a_k)})}
		\right) }{\log 2}, \\
	A &:= 7.22 \ \mbox{and} \ B:=2.
\end{align*}  
For each $\mu$ we search for a convergent $p/q$ to $\gamma$ with $q>6M$ such that $\eps=\left\|\mu q \right\| - M \left\|\gamma q \right\| >0$. If we are successful, we obtain $a_1-a_{k+1} \leq \frac{\log (Aq/\eps)}{\log B}$ for that $\mu$. For some $\mu$'s however, it is impossible to find such an $\eps$ because of linear dependencies between $\gamma$, $1$ and $\mu$. We will treat these cases separately to obtain overall small upper bounds for $a_1-a_{k+1}$.

Figure~\ref{fig:BD-mitSchranken} shows all cases of the proof and all steps we need to do. We proceed top to bottom, that is, we handle the situations in the following order: 1.0, 1.1, 2.1, 1.2, 2.2, 1.3, 2.3, 1.4, 2.4. In each situation we check all possible parameters, i.e. each parameter runs up to a small bound that we have obtained in an earlier step.
In situations where the bounds of the parameters depend on the case, we check the largest possible number of parameters. E.g. in Situation 2.2 we set the bound for $n-m$ as the maximum of the bounds for $n-m$ obtained in Situation 1.0 and Situation 1.1 and we set the bound for $a_1-a_2$ as the maximum of the bounds for $a_1-a_2$ obtained in Situation 1.0 and Situation 2.1. Note that we always have the conditions $n-m\geq2$ and $a_1-a_2<a_1-a_3<a_1-a_4<a_1-a_5$. Running all these computations on a computer, we obtain the results presented in Figure~\ref{fig:BD-mitSchranken}, where we ignored some special cases. The computations were executed using Sage \cite{sagemath} and
took less than six hours on a computer with an Intel Core i7-8700 using
six cores.

\begin{figure}
\begin{tikzpicture}[auto,node distance=1.5cm]
%%%%%%%%%%%% Sit1
\mybox{sit10}{}{1.0}{none}{$a_1-a_2 \leq$ \textbf{\bIdOa} \newline or $n-m \leq$ \textbf{\bIdOn}}

\mybox{sit11}{below =of sit10}{1.1}
{$a_1-a_2 \leq \mIdIaII$}
{$a_1-a_3 \leq$ \textbf{\bIdIa} \newline or $n-m \leq$ \textbf{\bIdIn}}

\mybox{sit12}{below =of sit11}{1.2}
{$a_1-a_2 \leq \mIdIIaII$ \newline $a_1-a_3 \leq \mIdIIaIII$}
{$a_1-a_4 \leq$ \textbf{\bIdIIa} \newline
or $n-m \leq$\textbf{\bIdIIn}}

\mybox{sit13}{below =of sit12}{1.3}
{$a_1-a_2 \leq \mIdIIIaII$ \newline $a_1-a_3\leq \mIdIIIaIII$ \newline $a_1-a_4 \leq \mIdIIIaIV$}
{$a_1-a_5 \leq$ \textbf{\bIdIIIa} \newline or $n-m \leq$ \textbf{\bIdIIIn}}

\mybox{sit14}{below =of sit13}{1.4}{$a_1-a_2 \leq \mIdIVaII$ \newline $a_1-a_3\leq \mIdIVaIII$ \newline $a_1-a_4\leq \mIdIVaIV$ \newline $a_1-a_5\leq \mIdIVaV$}{$n-m\leq$ \textbf{\bIdIVn}}

%%%%%%%%%% Sit2
\mybox{sit21}{right = of sit11}{2.1}{$n-m \leq \mIIdIm$}{$a_1-a_2 \leq$ \textbf{\bIIdI}}

\mybox{sit22}{right = of sit12}{2.2}{$n-m \leq \mIIdIIm$ \newline $a_1-a_2 \leq \mIIdIIaII$}
{$a_1-a_3 \leq$ \textbf{\bIIdII}}

\mybox{sit23}{right = of sit13}{2.3}{$n-m \leq \mIIdIIIm$ \newline $a_1-a_2 \leq \mIIdIIIaII$ \newline $a_1-a_3 \leq \mIIdIIIaIII$}{$a_1-a_4 \leq$ \textbf{\bIIdIII}}

\mybox{sit24}{right = of sit14}{2.4}{$n-m\leq \mIIdIVm$ \newline $a_1-a_2\leq \mIIdIVaII$ \newline $a_1-a_3\leq \mIIdIVaIII$ \newline $a_1-a_4\leq \mIIdIVaIV$}{$a_1-a_5\leq$ \textbf{\bIIdIV}}

%%%%%%%%%%%% Pfeile
%% nach rechts
\draw [-latex,thick] (sit10) --  node [fill=white, anchor=center, pos=0.5] {\small{\qquad $2^{-(a_1-a_2)} < \alpha^{-(n-m)}$}} (sit21) ;
\draw [-latex,thick] (sit11) --  node [fill=white, anchor=center, pos=0.5] {\small{\quad $2^{-(a_1-a_3)} < \alpha^{-(n-m)}$}} (sit22) ;
\draw [-latex,thick] (sit12) --  node [fill=white, anchor=center, pos=0.5] {\small{\quad $2^{-(a_1-a_4)} < \alpha^{-(n-m)}$}} (sit23) ;
\draw [-latex,thick] (sit13) --  node [fill=white, anchor=center, pos=0.5] {\small{\quad $2^{-(a_1-a_5)} < \alpha^{-(n-m)}$}} (sit24) ;
%% sit1 nach unten
\draw [-latex,thick] (sit10) --  node [fill=white, anchor=center, pos=0.5] {\small{$2^{-(a_1-a_2)} \geq \alpha^{-(n-m)}$}} (sit11) ;
\draw [-latex,thick] (sit11) --  node [fill=white, anchor=center, pos=0.5] {\small{$2^{-(a_1-a_3)} \geq \alpha^{-(n-m)}$}} (sit12) ;
\draw [-latex,thick] (sit12) --  node [fill=white, anchor=center, pos=0.5] {\small{$2^{-(a_1-a_4)} \geq \alpha^{-(n-m)}$}} (sit13) ;
\draw [-latex,thick] (sit13) --  node [fill=white, anchor=center, pos=0.5] {\small{$2^{-(a_1-a_5)} \geq \alpha^{-(n-m)}$}} (sit14) ;
%% Sit2 nach unten
\draw [-latex,thick] (sit21) --  node [right] {} (sit22) ;
\draw [-latex,thick] (sit22) --  node [right] {} (sit23) ;
\draw [-latex,thick] (sit23) --  node [right] {} (sit24) ;

\end{tikzpicture}

\caption{Results of the reduction of bounds}
\label{fig:BD-mitSchranken}
\end{figure}
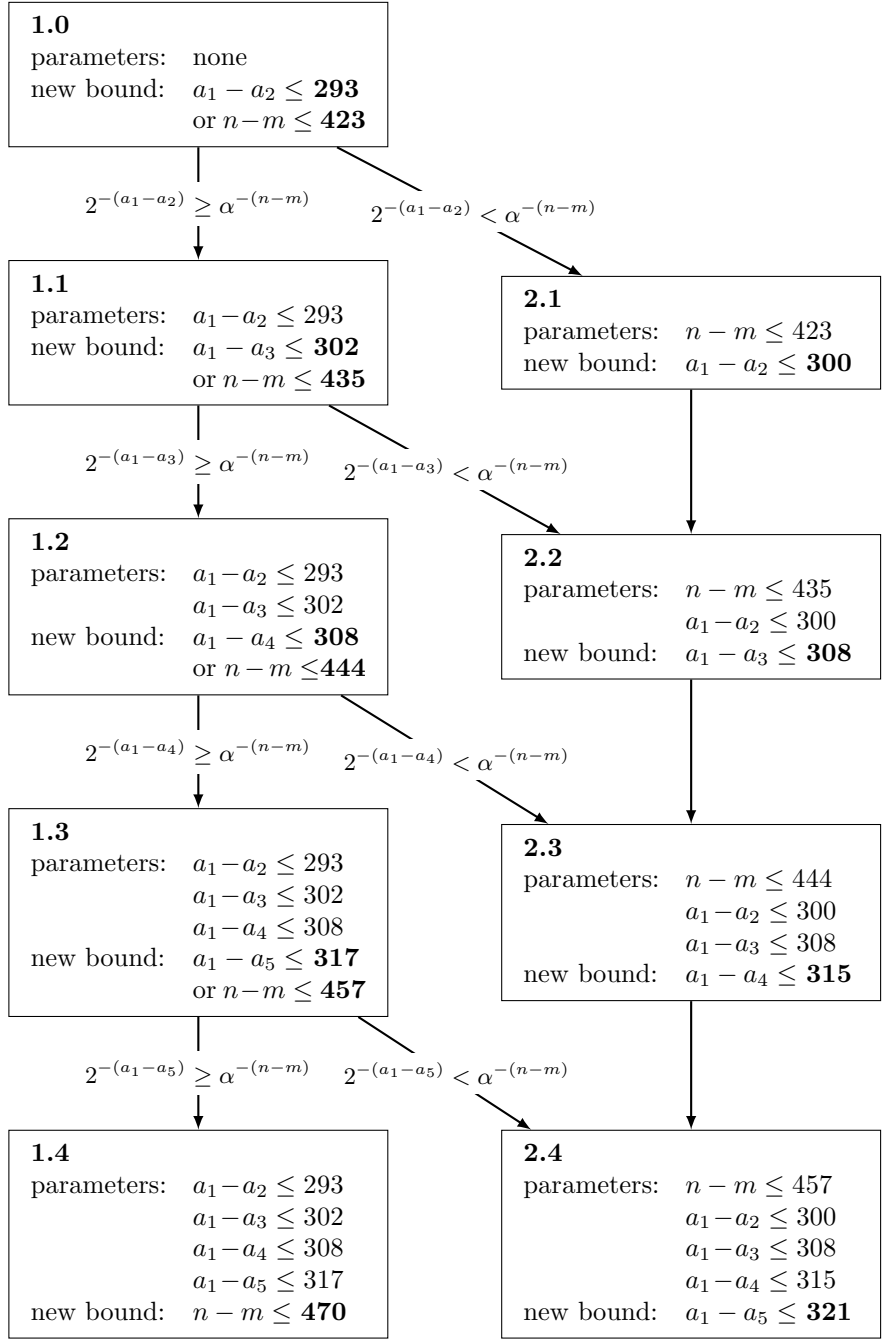
In Situations 2.1, 2.2, 2.3 and 2.4 we cannot apply Theorem~\ref{thm:baker-davenport} for all instances because of linear dependencies. 
These special cases are the following.
\begin{description}
	\item[Situation 2.1] $n-m = 2$ and $n-m=6$
	\item[Situation 2.2] $(n-m,a_1-a_2)=(10,2)$ and $(n-m,a_1-a_2)=(18,4)$
	\item[Situation 2.3] $(n-m,a_1-a_2,a_1-a_3)=(14,1,3)$
	\item[Situation 2.4] $(n-m,a_1-a_2,a_1-a_3,a_1-a_4)=(22,2,3,6)$ and $(n-m,a_1-a_2,a_1-a_3,a_1-a_4)=(30,8,4,3)$.
\end{description}
We now treat these cases separately and show that in each case we can still obtain $a_1-a_{k+1}\leq 293$.
We consider 
\[
	\mu'_{n-m,a_1-a_2,\dots, a_1-a_k}:= \frac{1+\alpha^{-(n-m)}}{\sqf\left(1+ 2^{-(a_1-a_2)} +\dots + 2^{-(a_1-a_k)}\right)}.
\] 

For each special case it is possible to rewrite $\mu'$ as the product of a power of 2 and a power of $\alpha$. For instance, $\mu_2'=\frac{1+\alpha^{-1}}{\sqrt5}=\alpha^{-1}$ and $\mu_6'=\frac{1+\alpha^{-6}}{\sqrt5}=2\alpha^{-3}$.
It turns out that each
$\mu'$ is of the form
\[
	\mu'=2^r \alpha^{-s} \beistr
	0\leq r \leq 9 \beistr
	1 \leq s \leq 15,
\]
and inequality \eqref{eq:BD_Ungl_Sit2} becomes
\[
	\left| n \gamma - a_1 + \log(2^r \alpha^{-s})/\log 2
	\right|
	\leq 
	7.22 \cdot 2^{-(a_1-a_{k+1})},
\]
where $\gamma = \frac{\log \alpha}{\log 2}$. Since
\[
	\left| n \gamma - a_1 + \frac{r\cdot\log 2 - s\cdot \log \alpha}{\log 2}
	\right|
	=
	\left| n \gamma - a_1 + r - s\gamma
	\right|
	= \left| (n-s)\gamma - (a_1-r) \right|,
\]
dividing by $n-s$ we obtain
\begin{equation}\label{eq:ungl:kettenbruch_nach_oben}
	\left|\gamma - \frac{a_1-r}{n-s}  \right|
	\leq 
	\frac{7.22}{2^{a_1-a_{k+1}}\cdot(n-s)}.
\end{equation}
If $a_1-a_{k+1}\leq 293$, then we do not need to do anything. 
Assume $a_1-a_{k+1}> 293$. Recall that by Proposition~\ref{pro:grosseSchranke} we have $n<1.54 \cdot 10^{85}$, so $2^{a_1-a_{k+1}}\geq 2^{294}>3.1\cdot 10^{88}>10^2 (n-s)$. Thus
\[
	\left|\gamma - \frac{a_1-r}{n-s}  \right|
	<
	\frac{1}{2\cdot (n-s)^2},
\]
which by the theory of continued fractions (see e.g. \cite[p. 47]{Baker_IntroductionToTheoryOfNumbers})
implies that $\frac{a_1-r}{n-s}$ is a convergent to $\gamma$, i.e. $\frac{a_1-r}{n-s}=\frac{p_j}{q_j}$ for some $j$. Since $q_{168}>2.7\cdot 10^{86}>n-s=q_j$, we have $j\in \{0,1,2,\dots,167\}$. Using a property of continued fractions (see e.g. \cite[p. 47]{Baker_IntroductionToTheoryOfNumbers}) and  inequality \eqref{eq:ungl:kettenbruch_nach_oben} we get
\[
	\frac{1}{q_j(q_j+q_{j+1})}
	<
	\left|\gamma - \frac{q_j}{q_j} \right|
	=
	\left|\gamma - \frac{a_1-r}{n-s}  \right|
	<
	\frac{7.22}{2^{a_1-a_{k+1}}\cdot(n-s)}
	=	
	\frac{7.22}{2^{a_1-a_{k+1}}\cdot q_j},
\]
which implies
\[
	a_1-a_{k+1}
	<
	\frac{\log 7.22 + \log(q_j+q_{j+1})}{\log 2}
	\leq
	\frac{\log 7.22 + \log(q_{167}+q_{168})}{\log 2}
	<
	291.
\]
Thus $a_1-a_{k+1}\leq 293$ and
the bounds in Figure~\ref{fig:BD-mitSchranken} are correct. Taking the maximum over the bounds for each expression $n-m$, $a_1-a_2$, $a_1-a_3$, $a_1-a_4$ and $a_1-a_5$, we obtain the bounds from Proposition~\ref{pro:Baker-Davenport}.
\end{proof}

Next, we reduce the bound for $a_5$ obtained in Proposition~\ref{pro:Bugeaud} using a $p$-adic reduction method.

\begin{pro}\label{pro:Pethoe-DeWeger}
	Assume that $(n,m,a_1,a_2,a_3,a_4,a_5)$ is a solution to Problem~\ref{probl:diophantische_gleichung}. Then we have that $a_5\leq295$.
\end{pro}

\begin{proof}
We use a $p$-adic reduction method that Pink and Ziegler \cite[section 7.2]{Pink_Ziegler} used to resolve Diophantine equations of the form $u_n+u_m=wp_1^{z_1}\dots p_s^{z_s}$. The reduction method is based on an algorithm due to Peth\H{o} and de Weger \cite[Algorithm A]{Pethoe_deWeger}.

We recall inequality \eqref{eq:Bugeaud_abschaetzung} from the proof of Proposition~\ref{pro:Bugeaud}:
\begin{equation*}\label{eq:pink_1}
	a_5 \leq v_2\left(
		\left(\frac{\alpha}{\beta}\right)^m-
		\frac{\beta^{n-m}+1}{\alpha^{n-m}+1}
	\right)+1.
\end{equation*}
Since $v_2\left(\frac{\alpha^{n-m}+1}{\beta^{n-m}+1}\right)=0$, we have
\[
	v_2\left(
		\left(\frac{\alpha}{\beta}\right)^m-
		\frac{\beta^{n-m}+1}{\alpha^{n-m}+1}
	\right)
	=
	v_2\left(
		\frac{\alpha^{n-m}+1}{\beta^{n-m}+1}		
		\left(\frac{\alpha}{\beta}\right)^m
		-1
	\right).
\]
We set 
\[
	t:=n-m
	\quad \mbox{and} \quad
	\tau(t):=\frac{\alpha^t+1}{\beta^t+1}.
\]
With this notation we have
\[
	a_5-1
	\leq	
	v_2\left(
		\tau(t)\left(\frac{\alpha}{\beta}\right)^m
		-1
	\right).
\]
The aim is to find an upper bound for the expression on the right hand side. Therefore we consider each possible $t=n-m=2,3,\dots , 470$.

The idea is to apply the $2$-adic logarithm in order to get rid of $m$ as an exponent. 
Note that the $p$-adic logarithm (defined on the non-zero complex $p$-adic numbers)
has the standard property $\log_p(x^m)=m \log_p(x)$. Moreover, if $v_p(x)>1$, then $v_p\left(\log_p(1+x)\right)=v_p(x)$ (see e.g. \cite[II.2.4]{Smart}).

If 
$
	v_2\left(\tau(t)\left(\frac{\alpha}{\beta}\right)^m -1 \right)
	\leq 1
$,
then $a_5\leq 2$ and we are done. 

Assume that 
$
	v_2\left(\tau(t)\left(\frac{\alpha}{\beta}\right)^m -1 \right)>1.
$
Then we can  use the properties of the $p$-adic logarithm:
\begin{align}\label{eq:pink_2}
	a_5-1 
	\leq	
	v_2\left(
		\tau(t)\left(\frac{\alpha}{\beta}\right)^m
		-1
	\right)
	&=
	v_2\left(
		\log_2 \left(
		\tau(t)\left(\frac{\alpha}{\beta}\right)^m
		\right)
	\right)\\
	&=
	v_2\left(
		\log_2  \tau(t) 
		-
		m \cdot \log_2 \left(\beta/\alpha\right)
	\right) \nonumber \\
	&=
	v_2\left(\log_2  \tau(t) \right)
	+
	v_2\left(
		1
		-
		m \cdot 
		\frac{
			\log_2 \left(\beta/\alpha\right)}
	{\log_2 \tau(t) }
	\right), \nonumber
\end{align}
where in the last step we multiplied and divided the expression by $\log_2\tau(t)$.
Note that if $\log_2\tau(t)$ were zero, the case would have to be treated separately, but in our computations this does not occur.
With the help of Sage \cite{sagemath} we compute (for each $t$) the expression
\[
	\zeta(t) := \frac{\log_2\tau(t)}{\log_2(\beta/\alpha)}
	= u_0+u_1\cdot 2 + u_2 \cdot 2^2 + \dots
	\quad \in \Q_2.
\]
Indeed, $\zeta(t)$ lies in $\Q_2$ for all $t$. 
This is because $\alpha$ and $\beta$ are conjugate in $\Q(\sqrt{5})$ and therefore $\log_2\tau(t) \in \sqrt{5}\Q_2$ and $\log_2(\beta/\alpha)\in \sqrt{5}\Q_2$ and so the quotient lies in $\Q_2$. It is, however, not obvious why $v_2(\zeta(t))\geq 0$. Yet, it turns out, that all considered $\zeta$'s are of the above form.

We choose $r$ smallest possible such that $2^r>M=1.54\cdot 10^{85}$, that is, $r=283$. We choose the unique integer $0\leq m_0 < 2^r$ that fulfills $m_0 \equiv \zeta \pmod{2^r}$, i.e. $m_0=u_0 + u_1 \cdot 2 + u_2 \cdot 2^2 + \dots + u_{282} \cdot 2^{282}$. 
 
Since $m<M<2^r$, by construction of $m_0$ we have
\begin{align*}
	a_5-1 
	&\leq
	v_2\left(\log_2 \left( \tau(t) \right)\right)
	+
	v_2\left(
		1
		-
		m \cdot 
		\frac{
			\log_2 \left(\beta/\alpha\right)}
	{\log_2 \left( \tau(t) \right)}
	\right)\\
	&\leq
	v_2\left(\log_2 \left( \tau(t) \right)\right)
	+
	v_2\left(
		1
		-
		m_0 \cdot 
		\frac{
			\log_2 \left(\beta/\alpha\right)}
	{\log_2 \left( \tau(t) \right)}
	\right).		
\end{align*}
Next, we let $R$ be the smallest index $\geq r$ such that $u_R\neq 0$. Note that in general such an integer does not necessarily exist. These cases can be treated separately but in our computations we always find such an integer.
Then we have
\begin{align*}
	a_5-1 
	&\leq
	v_2\left(\log_2 \left( \tau(t) \right)\right)
	+
	v_2\left(
		1
		-
		m_0 \cdot 
		\frac{
			\log_2 \left(\beta/\alpha\right)}
	{\log_2 \left( \tau(t) \right)}
	\right)\\
	&=
	v_2\left(\log_2 \left( \tau(t) \right)\right)
	+
	v_2\left(
		1
		-
		(\zeta - u_R\cdot2^R - u_{R+1}\cdot2^{R+1} - \dots)
		\cdot 
		\frac{
			\log_2 \left(\beta/\alpha\right)}
	{\log_2 \left( \tau(t) \right)}
	\right)\\
	&=
	v_2\left(\log_2 \left( \tau(t) \right)\right)
	+
	v_2\left(
		(u_R\cdot2^R+u_{R+1}\cdot2^{R+1}+\dots) \cdot 
		\frac{
			\log_2 \left(\beta/\alpha\right)}
	{\log_2 \left( \tau(t) \right)}
	\right)\\
	&=
	v_2\left(\log_2 \left( \tau(t) \right)\right)
	+
	R
	+
	v_2\left(
			\log_2 \left(\beta/\alpha\right)
			\right)
	-
	v_2\left(
			\log_2 \left(\tau(t)\right)
			\right)\\
	&=
	R
	+
	v_2\left(
			\log_2 \left(\beta/\alpha\right)
			\right)\\
	&=R+2
	.
\end{align*}
Using Sage \cite{sagemath} we run the computations for $t=2,3,\dots,470$. The largest $R$ to appear is $R_{max}=292$, therefore we obtain the bound $a_5\leq R_{max}+3=295$.
The computations took only a couple of seconds on a usual PC.
\end{proof}

\begin{proof}[Proof of Theorem \ref{thm:main}]
All the way (in Sections \ref{sec:largeUpperBound} and \ref{sec:reducing_bound}) we assumed that $n>1000$.
In Proposition~\ref{pro:Pethoe-DeWeger} we obtained $a_5\leq295$ and in Proposition~\ref{pro:Baker-Davenport} we obtained $a_1-a_5\leq 321$. Combined, these two results yield $a_1<616$. From inequality \eqref{eq:ineq_2hocha_groesser_alphahochn}
it follows that $n \log \alpha -\log 6 < a_1 \log 2$
and so $n<\frac{a_1 \log 2 + \log 6}{\log \alpha}<892$, which is a contradiciton to $n>1000$. Thus we must have $n\leq 1000$, which means that the solutions of Problem~\ref{probl:diophantische_gleichung} are exactly the solutions found in Proposition~\ref{pro:loesungen}.
\end{proof}

\section{Appendix - List of solutions for Problem \ref{probl:diophantische_gleichung}}
The solutions for Diophantine Equation~\eqref{eq:main} in Problem~\ref{probl:diophantische_gleichung} are displayed below.

\begingroup
	\allowdisplaybreaks
\begin{align*}
	&F_{ 9 }+F_{  7 }= 2^{ 5 }+ 2^{ 3 }+2^{ 2 }+2^{ 1 }+2^{ 0 }= 47 ,\\
&F_{ 11 }+F_{  3 }= 2^{ 6 }+ 2^{ 4 }+2^{ 3 }+2^{ 1 }+2^{ 0 }= 91 ,\\
&F_{ 11 }+F_{  5 }= 2^{ 6 }+ 2^{ 4 }+2^{ 3 }+2^{ 2 }+2^{ 1 }= 94 ,\\
&F_{ 11 }+F_{  8 }= 2^{ 6 }+ 2^{ 5 }+2^{ 3 }+2^{ 2 }+2^{ 1 }= 110 ,\\
&F_{ 12 }+F_{  7 }= 2^{ 7 }+ 2^{ 4 }+2^{ 3 }+2^{ 2 }+2^{ 0 }= 157 ,\\
&F_{ 12 }+F_{  10 }= 2^{ 7 }+ 2^{ 6 }+2^{ 2 }+2^{ 1 }+2^{ 0 }= 199 ,\\
&F_{ 13 }+F_{  2 }= 2^{ 7 }+ 2^{ 6 }+2^{ 5 }+2^{ 3 }+2^{ 1 }= 234 ,\\
&F_{ 13 }+F_{  4 }= 2^{ 7 }+ 2^{ 6 }+2^{ 5 }+2^{ 3 }+2^{ 2 }= 236 ,\\
&F_{ 13 }+F_{  6 }= 2^{ 7 }+ 2^{ 6 }+2^{ 5 }+2^{ 4 }+2^{ 0 }= 241 ,\\
&F_{ 14 }+F_{  8 }= 2^{ 8 }+ 2^{ 7 }+2^{ 3 }+2^{ 2 }+2^{ 1 }= 398 ,\\
&F_{ 14 }+F_{  11 }= 2^{ 8 }+ 2^{ 7 }+2^{ 6 }+2^{ 4 }+2^{ 1 }= 466 ,\\
&F_{ 15 }+F_{  2 }= 2^{ 9 }+ 2^{ 6 }+2^{ 5 }+2^{ 1 }+2^{ 0 }= 611 ,\\
&F_{ 15 }+F_{  4 }= 2^{ 9 }+ 2^{ 6 }+2^{ 5 }+2^{ 2 }+2^{ 0 }= 613 ,\\
&F_{ 15 }+F_{  6 }= 2^{ 9 }+ 2^{ 6 }+2^{ 5 }+2^{ 3 }+2^{ 1 }= 618 ,\\
&F_{ 15 }+F_{  10 }= 2^{ 9 }+ 2^{ 7 }+2^{ 4 }+2^{ 3 }+2^{ 0 }= 665 ,\\
&F_{ 16 }+F_{  5 }= 2^{ 9 }+ 2^{ 8 }+2^{ 7 }+2^{ 6 }+2^{ 5 }= 992 ,\\
&F_{ 16 }+F_{  14 }= 2^{ 10 }+ 2^{ 8 }+2^{ 6 }+2^{ 4 }+2^{ 2 }= 1364 ,\\
&F_{ 17 }+F_{  6 }= 2^{ 10 }+ 2^{ 9 }+2^{ 6 }+2^{ 2 }+2^{ 0 }= 1605 ,\\
&F_{ 17 }+F_{  7 }= 2^{ 10 }+ 2^{ 9 }+2^{ 6 }+2^{ 3 }+2^{ 1 }= 1610 ,\\
&F_{ 17 }+F_{  8 }= 2^{ 10 }+ 2^{ 9 }+2^{ 6 }+2^{ 4 }+2^{ 1 }= 1618 ,\\
&F_{ 18 }+F_{  2 }= 2^{ 11 }+ 2^{ 9 }+2^{ 4 }+2^{ 3 }+2^{ 0 }= 2585 ,\\
&F_{ 18 }+F_{  3 }= 2^{ 11 }+ 2^{ 9 }+2^{ 4 }+2^{ 3 }+2^{ 1 }= 2586 ,\\
&F_{ 18 }+F_{  7 }= 2^{ 11 }+ 2^{ 9 }+2^{ 5 }+2^{ 2 }+2^{ 0 }= 2597 ,\\
&F_{ 18 }+F_{  12 }= 2^{ 11 }+ 2^{ 9 }+2^{ 7 }+2^{ 5 }+2^{ 3 }= 2728 ,\\
&F_{ 19 }+F_{  2 }= 2^{ 12 }+ 2^{ 6 }+2^{ 4 }+2^{ 2 }+2^{ 1 }= 4182 ,\\
&F_{ 19 }+F_{  5 }= 2^{ 12 }+ 2^{ 6 }+2^{ 4 }+2^{ 3 }+2^{ 1 }= 4186 ,\\
&F_{ 19 }+F_{  8 }= 2^{ 12 }+ 2^{ 6 }+2^{ 5 }+2^{ 3 }+2^{ 1 }= 4202 ,\\
&F_{ 20 }+F_{  8 }= 2^{ 12 }+ 2^{ 11 }+2^{ 9 }+2^{ 7 }+2^{ 1 }= 6786 ,\\
&F_{ 20 }+F_{  17 }= 2^{ 13 }+ 2^{ 7 }+2^{ 5 }+2^{ 3 }+2^{ 1 }= 8362 ,\\
&F_{ 20 }+F_{  18 }= 2^{ 13 }+ 2^{ 10 }+2^{ 7 }+2^{ 2 }+2^{ 0 }= 9349 ,\\
&F_{ 22 }+F_{  2 }= 2^{ 14 }+ 2^{ 10 }+2^{ 8 }+2^{ 5 }+2^{ 4 }= 17712 ,\\
&F_{ 22 }+F_{  8 }= 2^{ 14 }+ 2^{ 10 }+2^{ 8 }+2^{ 6 }+2^{ 2 }= 17732 ,\\
&F_{ 22 }+F_{  11 }= 2^{ 14 }+ 2^{ 10 }+2^{ 8 }+2^{ 7 }+2^{ 3 }= 17800 ,\\
&F_{ 22 }+F_{  13 }= 2^{ 14 }+ 2^{ 10 }+2^{ 9 }+2^{ 4 }+2^{ 3 }= 17944 ,\\
&F_{ 22 }+F_{  16 }= 2^{ 14 }+ 2^{ 11 }+2^{ 8 }+2^{ 3 }+2^{ 1 }= 18698 ,\\
&F_{ 23 }+F_{  8 }= 2^{ 14 }+ 2^{ 13 }+2^{ 12 }+2^{ 2 }+2^{ 1 }= 28678 ,\\
&F_{ 23 }+F_{  10 }= 2^{ 14 }+ 2^{ 13 }+2^{ 12 }+2^{ 5 }+2^{ 3 }= 28712 ,\\
&F_{ 23 }+F_{  12 }= 2^{ 14 }+ 2^{ 13 }+2^{ 12 }+2^{ 7 }+2^{ 0 }= 28801.
\end{align*}
\endgroup

\bibliographystyle{abbrv}
\bibliography{Literatur_Dioph}
\end{document}